\documentclass[12pt]{amsart} 
\usepackage{amssymb}
\usepackage{amsmath}
\usepackage{latexsym}
\usepackage{amscd}
\usepackage{amsfonts}
\usepackage{pb-diagram}
\usepackage[all,pdftex]{xy}
\usepackage[mathcal,mathscr]{eucal}
\usepackage{amsthm}
\usepackage{epsfig}
\usepackage{url}
\usepackage{pdfsync}
\usepackage{epsfig}
\usepackage{bbm}
\usepackage[pdftex,colorlinks]{hyperref}
\usepackage{pifont}
\usepackage{amsbsy}
\usepackage{epstopdf}
\usepackage{times}
\usepackage{epic,eepic}

\newtheorem{theorem}{Theorem}[section]

\newtheorem{lemma}[theorem]{Lemma}
\newtheorem{corollary}[theorem]{Corollary}

\theoremstyle{definition}
\newtheorem{definition}[theorem]{Definition}
\newtheorem{remark}[theorem]{Remark}
\newtheorem{example}[theorem]{Example}

\numberwithin{equation}{section}

\DeclareMathOperator{\End}{End}
\DeclareMathOperator{\Ker}{Ker}

\DeclareMathOperator{\id}{id}

\DeclareMathOperator{\Spec}{Spec}

\DeclareMathOperator{\Imag}{Im}

\newcommand\bR{{\mathbb R}}

\newcommand{\bV}{{{\mathbb V}}}

\newcommand{\ra}{\rightarrow}

\begin{document}
\title{Extended Kato inequalities for conformal operators}

\vspace{0.5cm}

\author{Daniel Cibotaru}
\email{daniel@mat.ufc.br}
\author{Matheus Vieira}
\email{matheus.vieira@ufes.br}

\begin{abstract} We prove, for a class of first order differential operators containing generalized gradients, the  Dirac and Penrose twistor operators, a family of Kato inequalities that interpolates between the classical and the refined Kato inequalities. For the Hodge-de-Rham operator we get a more detailed result. As applications, we obtain various Kato inequalities from the literature.
\end{abstract}

\subjclass[2010]{Primary 47A63, Secondary 53B20}
\pagestyle{empty}
\maketitle
\tableofcontents
 \pagenumbering{arabic}
\section{Introduction}

Refined Kato inequalities have been used in Riemannian geometry and gauge theory for a long time. Recall that the classical Kato inequality says that for any section $\phi$ of a Riemannian or Hermitian vector bundle endowed with a metric connection $\nabla$, the following inequality holds at points where $\phi\neq 0$:
\[|\nabla \phi|\geq |d|\phi||.\]
When the section $\phi$ is a solution of a first order injectively elliptic system, the inequality can be refined by multiplying the right hand side by a constant $K>1$. The constant $K$ depends on the operator involved. An important observation due to Bourguignon \cite{Bo}, which led to the systematic study of the best such constants in \cite{B2000,CGH2000}, is that for a class of geometric operators called Stein-Weiss one can derive a simple relation between $K$ and the ellipticity constant of the symbol of the operator. The latter can be entirely determined from representation theoretic data.

We recall some results in chronological order.

In \cite{SSY1975}, Schoen-Simon-Yau proved that if $M^{n}$ is a
minimal hypersurface in $R^{n+1}$ then
\[
\left|\nabla A\right|^{2}\geq\left(1+\frac{2}{n}\right)\left|d\left|A\right|\right|^{2},
\]
where $A$ is the second fundamental form.

In \cite{BKN1989}, Bando-Kasue-Nakajima proved that if $M^{4}$ is an Einstein manifold then
\[
\left|\nabla Rm\right|^{2}\geq\left(1+\frac{2}{3}\right)\left|d\left|Rm\right|\right|^{2}.
\]
where $Rm$ is the Riemann curvature.

In \cite{R1991}, Rade proved that if $A$ is a Yang-Mills connection
on $R^{4}$ then
\[
\left|\nabla F\right|^{2}\geq\left(1+\frac{1}{2}\right)\left|d\left|F\right|\right|^{2},
\]
where $F$ is the curvature of $A$.

In \cite{GL1999} and \cite{Y2000}, Gursky-Lebrun and Yang proved
independently that if $M^{4}$ is an Einstein manifold then
\[
\left|\nabla W^{+}\right|^{2}\geq\left(1+\frac{2}{3}\right)\left|d\left|W^{+}\right|\right|^{2},
\]
where $W$ is the Weyl curvature and $W^{+}=\frac{1}{2}\left(W+*W\right)$.

In \cite{F2001}, Feehan proved that if $M^{n}$ is a Riemannian manifold
and $\phi$ is a harmonic spinor on $M^{n}$ then
\[
\left|\nabla\phi\right|^{2}\geq\left(1+\frac{1}{n-1}\right)\left|d\left|\phi\right|\right|^{2}.
\]

In \cite{B2000} and \cite{CGH2000}, Branson and Calderbank-Gauduchon-Herzlich
developed different tools to determine the ellipticity constant of the symbol of a Stein-Weiss operator from representation theoretic data. For example, Theorem 6.3 in \cite{CGH2000} (see also \cite{CZ2012} for an alternative proof) proves that if $M^{n}$
is a Riemannian manifold and $\omega$ is a harmonic $k$-form on
$M^{n}$ then
\[
\left|\nabla\omega\right|^{2}\geq\left(1+\min\left\{ \frac{1}{k},\frac{1}{n-k}\right\} \right)\left|d\left|\omega\right|\right|^{2}.
\]



In more recent years, Kato type inequalities involving additional quantities, typically embodied in the form of the squared norm of a first order differential operator, have been used both in geometry and in gauge theory. We give some examples of what we call extended Kato inequalities.

In \cite{CM2012}, Colding-Minicozzi proved that if $M^{n}$ is a
hypersurface in $R^{n+1}$ then
\[
\left|\nabla A\right|^{2}+\frac{2n}{n+1}\left|\nabla H\right|^{2}\geq\left(1+\frac{2}{n+1}\right)\left|d\left|A\right|\right|^{2},
\]
where $A$ is the second fundamental form and $H$ is the mean curvature.

In \cite{CJS2014}, Chen-Jost-Sun proved that if $M^{n}$ is a Riemannian
manifold and $\phi$ is a section on a Dirac bundle on $M^n$ then
\[
\left|\nabla\phi\right|^{2}+c\left|D\phi\right|^{2}\geq\left(1+\frac{c}{1+\left(n-1\right)c}\right)\left|d\left|\phi\right|\right|^{2},
\]
where $D$ is the Dirac operator.


In \cite{SU2019}, Smith-Uhlenbeck proved that if $M^{n}$ is a Riemannian
manifold and $\omega$ is a vector valued two-form on $M^{n}$ then
\[
\left|\nabla\omega\right|^{2}+\left|d^\nabla \omega\right|^{2}+\left|(d^\nabla)^{*}\omega\right|^{2}\geq\left(1+\frac{1}{n-1}\right)\left|d\left|\omega\right|\right|^{2}.
\]
They also proved that if $M^{n}$ is a Riemannian manifold and $\omega$
is a vector valued one-form on $M^{n}$ then
\[
\left|\nabla\omega\right|^{2}+\left|d^\nabla \omega\right|^{2}+\left|(d^\nabla)^{*}\omega\right|^{2}\geq\left(1+\frac{1}{n}\right)\left|d\left|\omega\right|\right|^{2}.
\]
The results are proved in the case of an open subset of $R^{n}$ but
they are written in a way that can be adapted to arbitrary Riemannian
manifolds. In \cite{F2023}, Fadel improved the constant of Smith-Uhlenbeck for
one-forms on three-dimensional manifolds.

Some other papers that make use of refined Kato inequalities are \cite{CB2011,GKS2018,S1991,TV2005a,TV2005b}. 

The purpose of this note is twofold. The first is to present an adequate theoretical context that allows us to prove extended Kato inequalities that reduce to refined Kato inequalities when the section is a solution of a first order injectively elliptic operator. The novelty here is that we consider first order operators whose symbol is a conformal projection (see Definition \ref{confdef}), which allows for a unified treatment of operators such as generalized gradient and Dirac operators. The second is to produce, within this theoretical background, general results of which some of the above examples are particular cases.

In our first general result we obtain an extended Kato inequality for conformal injectively elliptic first order linear differential operators (see Definition \ref{confoperator}).

\begin{theorem}
\label{thm:foldo} Let $D :\Gamma(E)\ra \Gamma(F)$ be a conformal, injectively elliptic, first order differential operator between Riemannian vector bundles $E,F \ra M$ with ellipticity constant $\epsilon$ and conformity factor $\rho$. Fix a constant $c\geq 0$. Let $\phi$
be a section of $E$ and let $x$ be a point in $M$ such that $\phi(x)\neq0$.
Then at $x$ the following pointwise inequality holds
\[
|\nabla\phi|^{2}+c|D\phi|^{2}\geq\left(1+\tilde{c}\right)|d|\phi||^{2},
\]
where
\[
\tilde{c}=\begin{cases}
\frac{\epsilon}{\rho^2-\epsilon} & if\,\,\,\left(D\phi\right)_{x}=0,\\
\frac{\epsilon c}{1+\left(\rho^2-\epsilon\right)c} & if\,\,\,\left(D\phi\right)_{x}\neq0.
\end{cases}
\]
\end{theorem}

Note that taking $c=0$ we get the classical Kato inequality, while taking $c=\infty$, which corresponds to the border case when $(D\phi)_x=0$ (with the convention $\infty\cdot 0=0$), we get the refined Kato inequality with optimal constant (see also Theorem \ref{Dinjelliptic}).

The classical Hodge-deRham operator $d+d^ *$ is the sum of two operators which, taken separately, are not injectively elliptic.  In the next result we obtain a Kato inequality for vector valued differential forms which are not necessarily closed or coclosed. It should be viewed as a refinement of Theorem \ref{thm:foldo} for a situation where the symbol of the injectively elliptic operator has two conformal projections as components. It is unclear whether this refinement can be extended to other Branson's  elliptic generalized gradients \cite{Br1} for two element sets.

\begin{theorem}
\label{thm:hodge} Let $\nabla$ be a metric connection on a Riemannian vector bundle $E \ra M^n$. Fix constants $c\geq 0$ and $c_{*}\geq 0$. Let $\phi$ be a section of $\Lambda^{k}T^{*}M\otimes E$ and let $x$ be a point in $M$ such that $\phi(x)\neq0$. Then at $x$ the following pointwise inequality holds
\[
|\nabla\phi|^{2}+c|d^\nabla \phi|^{2}+c_{*}|(d^{\nabla})^*\phi|^{2}\geq\left(1+\min\left\{ \tilde{c},\tilde{c}_{*}\right\} \right)|d|\phi||^{2},
\]
where
\[
\tilde{c}=\begin{cases}
\frac{1}{k} & if\,\,\,\left(d^\nabla \phi\right)_{x}=0,\\
\frac{c}{1+kc} & if\,\,\,\left(d^\nabla \phi\right)_{x}\neq0,
\end{cases}\,\,\,\,\,\,\,\,\,\,\tilde{c}_{*}=\begin{cases}
\frac{1}{n-k} & if\,\,\,\left((d^\nabla)^{*}\phi\right)_{x}=0,\\
\frac{c_{*}}{1+\left(n-k\right)c_{*}} & if\,\,\,\left((d^\nabla)^{*}\phi\right)_{x}\neq0.
\end{cases}
\]
On the left hand side, $d^\nabla$ is the covariant exterior derivative of
$\nabla$ and $(d^\nabla)^*$ is the adjoint of $d^\nabla$.
\end{theorem}
We would like to point out:

(i) Since $\frac{1}{k}>\frac{c}{1+kc}$ and $\frac{1}{n-k}>\frac{c_{*}}{1+\left(n-k\right)c_{*}}$
we see that the conclusion of Theorem \ref{thm:hodge} is stronger
for closed or co-closed forms. In particular, for harmonic forms
we recover the known Kato inequality (see e.g. \cite{CZ2012}), but
in the more general setting of vector valued forms.

(ii) Note that \[\lim_{c\to\infty}\frac{c}{1+kc}=\frac{1}{k}\qquad \mbox{and} \qquad
\lim_{c_{*}\to\infty}\frac{c_{*}}{1+\left(n-k\right)c_{*}}=\frac{1}{n-k}.\] In this sense the extended Kato inequalities interpolate between the classical $c=c_*=0$ and the refined $c=c_*=\infty$ cases. A similar observation applies to Theorem \ref{thm:foldo}.

\begin{remark}
Using our main results, we will recover many sharp inequalities described above. On the other hand, we were unable to recover some inequalities, such as the one of Colding-Minicozzi (Lemma 10.2 in \cite{CM2012}). To see this, note that the second fundamental form $A$ can be seen as a linear operator $A:TM \to TM$. As vector valued one-form, by the Codazzi equation, the second fundamental form satisfies $dA = 0$ and $d^{*}A=-dH$. Applying Theorem \ref{thm:hodge} with $\phi = A$, $k=1$ and $c^{*}=\frac{2n}{n+1}$ we get the extended Kato inequality
$$\left|\nabla A\right|^{2}+\frac{2n}{n+1}\left| \nabla H\right|^{2}\geq\left(1+\frac{2}{2n-1+\frac{1}{n}}\right)\left|\nabla\left|A\right|\right|^{2}.$$
This inequality is less sharp than the one of Colding-Minicozzi. We would  like to mention that in certain situations it is sufficient to have a Kato inequality with a non-sharp constant. For example, the above inequality suffices for the proof of Proposition 10.4 in \cite{CM2012}, which was the goal of their inequality.

The inequalities are sharp in the refined case  by Branson's Theorem 7 from \cite{B2000}.
\end{remark}

The paper is organized as follows. In Section \ref{sect2} we introduce the class of conformal first order linear differential operators and prove an $"L^1"$ extended Kato inequality.  In Section \ref{sect3} we prove Theorem \ref{thm:foldo} and in Section \ref{sect4} we prove Theorem \ref{thm:hodge}. Finally, in Section \ref{sect5} we give some applications.

While this paper was under review, particular cases of Theorem \ref{thm:hodge} appeared in \cite{CS}, albeit with a different  technique.

\section{Conformal operators of first order}\label{sect2}
\begin{definition}\label{confdef} A surjective linear operator $P:V\ra W$ between inner product vector spaces is called a conformal projection if $PP^ *=\rho^ 2\id_{W}$ for some  constant $\rho>0$, called the  conformity factor. The operator $P$ is called an orthogonal projection if $\rho=1$.
\end{definition}

\begin{remark}
A classical orthogonal projection $P:V\ra V$ such that $P=P^ *$ and $P^ 2=P$ is an orthogonal projection in the sense of Definition \ref{confdef} if we restrict the codomain of $P$ to $\Imag P$.

Conversely, if we have an orthogonal projection $P:V \ra W$ in the sense of Definition \ref{confdef} then $P^ *P:V\ra V$ an orthogonal projection in the classical sense. In other words, after identifying $W$ with $P^ *(W)$ via $P^*$ (which is an isometric injection), we can think of $P$ as the classical orthogonal projection onto the subspace $P^ *(W)$ of $V$.

With this definition, isometric isomorphisms are particular cases of orthogonal projections.

The adjoint of a conformal projection is a conformal immersion. 
\end{remark}

We fix an inner product vector space $(V,g)$ of dimension $n$ with orthonormal basis $\{e_1,\ldots, e_n\}$ and dual basis $\{e_1^ *,\ldots, e_n^ *\}$.

\begin{example}\label{firstex}

\item[(1)] Clifford multiplication. Let $S$ be a hermitian vector space and let $c:V\ra \End(S)$ be a unitary representation of the Clifford algebra of $(V,g)$, i.e. a linear morphism such that
\[ c(v)\circ c(w)+c(w)\circ c(v)=- 2 g(v,w).\]
The induced morphism $c:V\otimes S\ra S$ is a conformal projection. Indeed, the adjoint $c^*:S\ra V\otimes S$ is given by
\[ c^*(s)=\sum_{i=1}^ne_i\otimes c_{e_i}^ *(s),\]
where $c_{e_i}^ *=-c_{e_i}$ is the adjoint of $c_{e_i}:S\ra S$. Therefore
\[c(c^*(s))=-\sum_{i=1}^ nc_{e_i}c_{e_i}(s)=ns.\]

\item[(2)] Exterior multiplication. Let $\epsilon:V^ *\otimes \Lambda^{k}V^ *\ra \Lambda^ {k+1}V^ *$ be the exterior multiplication given by $\epsilon(f\otimes \omega)=f\wedge \omega$. The exterior multiplication is a conformal projection. Indeed, the adjoint of $\epsilon_f:\Lambda^kV^*\ra \Lambda^ {k+1}V^*$ is the interior multiplication $\iota_{f^{\sharp}}:\Lambda^ {k+1}V^ *\ra \Lambda^ kV^ *$. Then the adjoint $\epsilon^ *:\Lambda^ {k+1}V^ *\ra V^ *\otimes \Lambda^kV^ * $ is given by
\[\epsilon^ *(\eta)=\sum_{i} e_i^ *\otimes (\epsilon_{e_i^ * })^*(\eta)=\sum_{i=1}^ ne_i^ *\otimes \iota_{e_i}(\eta).\]
We can easily check on a basis of $\Lambda^{k+1}V^ *$ that
\[ \epsilon(\epsilon^ *(\eta))=(k+1)\eta.\]

\item[(3)] Interior multiplication. Let $\iota:V^ *\otimes \Lambda^{k}V^ *\ra \Lambda^ {k-1}V^ *$ be the interior multiplication given by $\iota(f\otimes \omega):=\iota_{f^ {\sharp}}(\omega)$. The interior multiplication is a conformal projection. Indeed, a similar computation as in the previous example gives
\[ \iota(\iota^ *(\eta))=(n-k+1)\eta.\]

\item[(4)] Symmetrization (symmetric product). Let $S^k(V^ *)$ be the space of $k$-multilinear symmetric morphisms on $V$ and let $\mathcal{S}:V^ *\otimes S^k(V^ *)\ra S^ {k+1}(V^ *)$ be the symmetrization morphism given by
\[\mathcal{S}(f\otimes A): (X_1,\ldots,X_{k+1})\ra \sum_{i=1}^ {k+1} f(X_i)A(X_1,\ldots,\hat{X_i},\ldots, X_{k+1}). \]
We choose the inner products on $V^ *\otimes S^k(V^ *)$ and $ S^ {k+1}(V^ *)$ induced from the canonical inner product of the tensor product $T^ {k+1}(V^ *)$. The symmetrization morphism is a conformal projection. Indeed, the adjoint $\mathcal{S}^* = (k+1)\iota$ where $\iota:S^ {k+1}V^ *\hookrightarrow V^ *\otimes S^k(V^ *)$ is the canonical inclusion of subspaces of $T^ {k+1}(V^ *)$. Then
\[\mathcal{S}\mathcal{S}^ *=(k+1)^ 2\id_{S^ {k+1}(V^ *)}. \]

\item[(5)] Contraction. Let $\mathcal{C}:V^*\otimes S^k(V^ *)\ra S^{k-1}(V^ *)$ be the contraction morphism given by
$$\mathcal{C}(f \otimes A):(X_1,\ldots,X_{k-1}) \ra A(f^ {\sharp},X_1,\ldots,X_{k-1}) .$$
The contraction morphism is a conformal projection. Indeed, denoting $\mathcal{C}_f:S^k(V^ *)\ra S^{k-1}(V^ *)$ with  $\mathcal{C}_f(A)= \mathcal{C}(f\otimes A)$ we have
\[(\mathcal{C}_f)^ *(B)=\frac{1}{k}\mathcal{S}(f\otimes B),\]
so
\[\mathcal{C}^ *(B)=\sum_{i=1}^ ne_i^ *\otimes (\mathcal{C}_{e_i^ *})^ *(B)=\frac{1}{k}\sum_{i=1}^ne_i^ * \otimes \mathcal{S}(e_i^ *\otimes B),\]
which implies
\[\mathcal{C}(\mathcal{C}^ *(B))=\frac{n+k-1}{k}B.\]

\item[(6)] Non-zero (intertwining) morphisms $\rho:V\ra W$ of orthogonal representations of a group $G$ with \emph{irreducible} target $W$ are conformal projections. Indeed, the morphism $\rho\rho^ *:W\ra W$ is a symmetric morphism of irreducible representations. By Schur lemma it is a multiple of the identity. Examples (2) and (3) above are particular cases for $G=SO(n)$.

\item[(7)] The differential of a Riemannian submersion is an orthogonal projection at every point.
\end{example}

We consider first order differential operators $D:\Gamma(E)\ra \Gamma(F)$ between Riemannian vector bundles $E,F\ra M$. These operators are compositions $D=\sigma(D) \circ \nabla$ of a bundle morphism $\sigma(D): \Gamma (T^ *M\otimes E) \ra \Gamma (F)$, called the principal symbol, and a metric connection $\nabla$ on $E$. In other words $D$ is the composition
\[D: \xymatrix{\Gamma(E)\ar[r]^ {\nabla\qquad} & \Gamma(T^ *M\otimes E) \ar[r]^ {\quad\sigma(D)} & \Gamma(F)}.\]

\begin{definition}\label{confoperator} A first order linear differential operator $D:\Gamma(E)\ra \Gamma(F)$ between Riemannian vector bundles $E,F\ra M$ is called a conformal operator if the principal symbol $P=\sigma(D): \Gamma (T^ *M\otimes E) \ra \Gamma (F)$ is a conformal projection at every point of $M$. The conformity factor of $D$ at a point $x$ is the conformity factor of $P_x$.
\end{definition}

\begin{example}

\item[(1)] Metric connections $\nabla:\Gamma(E)\ra \Gamma(T^ *M\otimes E)$ on a Riemannian vector bundle $E \ra M$ are conformal operators with conformity factor $\rho=1$.

\item[(2)] Dirac operators are conformal as follows from item (1) of Example \ref{firstex}. The conformity factor satisfies $\rho^ 2=n$, where $n$ is the dimension of the manifold.

\item[(3)] The exterior derivative $d$ and its adjoint $d^ *$ are conformal operators.

\item[(4)] Generalized gradients, a subclass of Stein-Weiss operators (see \cite{Br1,CGH2000}) are conformal. These operators arise as 
\[D:=P\circ \nabla:\Gamma(\bV(\lambda))\ra \Gamma(F),\]
where $\nabla$ is the Levi-Civita connection on the vector bundle $\bV(\lambda)$ associated to a representation $\lambda$ on $SO(n)$ or $Spin(n)$ and the principal bundle of orthonormal frames on a Riemannian manifold $M$ and $P:\bV(\lambda)\ra F$ is the bundle orthogonal projection induced by the projection  of $\tau\otimes \lambda$ onto an irreducible subrepresentation of $\tau\otimes \lambda$.

\item[(5)] Twisted generalized gradients are conformal.  They arise as follows: 
\[D\otimes E:\Gamma(\bV(\lambda)\otimes E)\ra \Gamma(F\otimes E) \qquad D\otimes E:=(P\otimes \id_E)\circ \nabla^ {\otimes},\]
where $\bV(\lambda)$ and $F$ are as before and $(E,\nabla^E)$ is a metric bundle with compatible connection. The connection $\nabla^ {\otimes}$ is the tensor product of the connections on $\bV(\lambda)$ and $E$. Examples contain $\frac{1}{\sqrt{k+1}}d^ {\nabla}$ and $\frac{1}{\sqrt{n-k+1}}(d^{\nabla})^*$ where $d^ {\nabla}$  is the exterior derivative acting on forms with values in $E$ and $(d^ {\nabla})^ *$ is its formal adjoint (see Lemma \ref{nL}).

\item[(6)] Penrose twistor operators are conformal. The symbol is the orthogonal projection onto the kernel of the Clifford multiplication.
\end{example}

\begin{definition} A first order linear differential operator $D:\Gamma(E)\ra \Gamma(F)$ between Riemannian vector bundles $E,F\ra M$ with principal symbol $P$ is called injectively elliptic at a point $x\in M$ if $P_{\xi}^ *P_{\xi}:E_x\ra E_x$ is bijective for every unit vector $\xi\in T_x^*M $. The number
\[ \epsilon = \inf_{|\xi|=1}\inf_{|v|=1} \|P_{\xi}(v)\|^2\]
is called the ellipticity constant of $D$ at $x$.
\end{definition}

Clearly $D$ is injectively elliptic if and only if $\epsilon>0$.

\begin{example}

\item[(1)] Metric connections $\nabla:\Gamma(E)\ra \Gamma(T^ *M\otimes E)$ on a Riemannian vector bundle $E \ra M$ are injectively elliptic operators with ellipticity constant $\epsilon=1$.

\item[(2)] Dirac operators are injectively elliptic with ellipticity constant $\epsilon=1$.

\item[(3)] Neither $d$ nor $d^ *$ are injectively elliptic except in the border degree cases $k=0$ and $k=n$ respectively. However, their sum $d+d^ *$ is injectively elliptic.

\item[(4)] The classification of injectively elliptic Stein-Weiss operators was done by Branson in \cite{Br1} (see also \cite{Pi}).  For Stein-Weiss operators the ellipticity constant is a constant function on $M$. This follows by noting that the symbol of a Stein-Weiss operator is up to conjugation the same in every fiber, and can be identified with the morphism of representations which is the orthogonal projection of $\tau\otimes \lambda$ onto a sum of irreducible subrepresentations. 

\item[(5)] The ellipticity constant of a twisted Stein-Weiss operator $D\otimes E$ equals the ellipticity constant of the original untwisted Stein-Weiss operator $D$. Therefore $D\otimes E$ is injectively elliptic if and only if $D$ is.  This follows from the following straightforward symbol identity 
\[\sigma_{\xi}((D\otimes E)^ *D\otimes E)=\sigma_{\xi}(D^ *D)\otimes \id_E.\]
 For example, the ellipticity constant for 
\[ \frac{1}{\sqrt{n-k+1}}(d^{\nabla})^ *+\frac{1}{\sqrt{k+1}}d^ {\nabla}\]
is the same as that of $\frac{1}{\sqrt{n-k+1}}d^ *+\frac{1}{\sqrt{k+1}}d$.

\item[(6)] Penrose twistor operators are injectively elliptic with $\epsilon=\frac{n-1}{n}$ (and conformity factor $\rho = 1$).
\end{example}

The following "$L^1$" extended Kato inequality is a rather immediate extension of the results that appeared in \cite{B2000,CGH2000} which makes clear the pointwise nature of refined Kato inequalities. 

\begin{theorem}\label{Dinjelliptic} Let $D:\Gamma(E)\ra\Gamma(F)$ be a conformal injectively elliptic operator  between Riemannian vector bundles $E,F \ra M$ with ellipticity constant $\epsilon$ and conformity factor $\rho$. Let $\phi$ be a section of $E$ and let $x$ be a point in $M$ such that $\phi(x)\neq0$. Then at $x$ the following pointwise inequality holds
\begin{equation}\label{fexK} \sqrt{\rho^ 2- \epsilon}|\nabla \phi|+ |D \phi|\geq  \rho|d|\phi||.
\end{equation}
In particular, if $(D\phi)_x=0$ then the following refined Kato inequality holds
\[|\nabla \phi|^ 2\geq \left(1+\frac{\epsilon}{\rho^2-\epsilon}\right)|d|\phi||^ 2. \]
\end{theorem}

\begin{proof}
If $|d|\phi||(x) = 0$ the result is trivial, so by continuity we can assume that $|d|\phi|| \ne 0$ in a neighborhood of $x$. We take the one-form $\xi_0:=\frac{d|\phi|^2}{|d|\phi|^2|}$, which is defined in a neighborhood of $x$ and has norm $1$. Then
\begin{equation}\label{equ5} |d|\phi||\cdot|\phi|=\langle \nabla_{\xi_0^{\sharp}}\phi,\phi\rangle=\langle \nabla \phi,\xi_0\otimes \phi \rangle ,
\end{equation}
where $(\cdot)^{\sharp}$ is the musical isomorphism. The justification is immediate:
\begin{equation}\label{mine} 2|\phi|\cdot|d|\phi||=|d|\phi|^2|=\langle d|\phi|^2,\xi_0\rangle=\xi_0^{\sharp}|\phi|^2=2\langle \nabla_{\xi_0^{\sharp}}\phi,\phi \rangle=2\langle \nabla \phi,\xi_0\otimes \phi \rangle.
\end{equation}
The operator $\frac{1}{\rho}P:T^ *M\otimes E\ra F$ is an orthogonal projection. Then there exists a projection $\pi^{F^ {\perp}}:T^ *M\otimes E\ra F^ {\perp}$ onto the orthogonal complement of $P^ *(F)$, and so
\[T^ *M\otimes E=P^ *(F)\oplus F^ {\perp}.\]
Correspondingly, we have \[\nabla \phi=\frac{1}{\rho}D\phi + \pi^{F^{\perp}} (\nabla \phi),\]
and so 
\begin{equation}\label{ineq1}\langle \nabla \phi,\xi_0\otimes \phi \rangle=\frac{1}{\rho}\langle D \phi ,\xi_0\otimes \phi\rangle+\langle \nabla \phi,\pi^{F^ {\perp}}(\xi_0\otimes \phi)\rangle.\end{equation}
We have:
\[\frac{1}{\rho^ 2}|P(\xi_0\otimes \phi)|^ 2+|\pi^ {F^ {\perp}}(\xi_0\otimes \phi)|^ 2=|\xi_0\otimes \phi|^ 2=|\phi|^ 2.\]
But the injective ellipticity condition implies that
\[|P(\xi_0\otimes \phi)|^ 2\geq \epsilon |\xi_0|^ 2|\phi|^ 2=\epsilon|\phi|^2, \]
and therefore
\[ \left(1-\frac{\epsilon}{\rho^ 2}\right)|\phi|^2 \geq |\pi^ {F^ {\perp}}(\xi_0\otimes \phi)|^2.\]
Using (\ref{ineq1}) and the Cauchy-Schwarz inequality we get
\begin{equation}\label{equa6}  \sqrt{1-\frac{\epsilon}{\rho^ 2}}|\nabla \phi|\cdot|\phi|+\frac{1}{\rho}|D(\phi)|\cdot|\phi|\geq |\langle \nabla \phi,\xi_0\otimes \phi \rangle|.\end{equation}
Now (\ref{equa6}) with (\ref{equ5}) gives (\ref{fexK}).
\end{proof}

\section{Extended, quadratic Kato inequalities} \label{sect3}

In this section we prove Theorem \ref{thm:foldo}. We start with some linear algebra facts.

\begin{lemma}\label{alglemm1} Let $P:U \ra F$ be a conformal projection between inner product spaces and suppose the target space has an orthogonal decomposition $F=F_1\oplus F_2$. Then the components $P_1$ and $P_2$ of $P$ relative to this decomposition are also conformal projections with the same conformity factor as $P$.
\end{lemma}

\begin{proof} It is straightforward to check that $P_i^ *:=P^ *\bigr|_{F_i}$ for $i=1,2$. Then for $f_i\in F_i$ we have
\[PP_i^ *(f_i)=PP^*(f_i)=\rho^ 2f_i\in F_i.\]
It follows that $PP_i^ *(f_i)=P_iP_i^ *(f_i)=\rho^2 f_i$.
\end{proof}

By analogy with the definition of the previous section, a linear operator $P:V\otimes E\ra F$ between inner product spaces is called injectively elliptic if $P_{\xi}:E\ra F$ is injective for every $\xi \in V\setminus \{0\}$, and the number
\[0<\epsilon:=\inf_{|\xi|=1}\inf_{|v|=1}\langle P_{\xi}^*P_{\xi}v,v\rangle\]
is called the ellipticity constant. 

In what follows we fix  an injectively elliptic conformal projection $P:V\otimes E\ra F$ with conformity factor $\rho$ and ellipticity constant $\epsilon$.

Let $L\subset V$ be a $1$-dimensional subspace. Let $F_1:=P(L\otimes E)$, let $F_2$ be the orthogonal complement of $F_1$ in $F$ and let $P=(P_1,P_2)$ be the components of $P$ relative this decomposition.
\begin{lemma} \label{alglem0}
\begin{itemize}
\item[(a)] Let $\tilde{P_1}:=P_1\bigr|_{L\otimes E}:L\otimes E\ra F_1$. Then $\inf\Spec(\tilde{P_1}\tilde{P_1}^ *)\geq \epsilon$.
\item[(b)] Let $W$ be the orthogonal complement of $L$ in $V$ and let $\hat{P_1}:=P_1\bigr|_{W\otimes E}:W \otimes E \ra F_1$. Then $\Spec(\hat{P_1}\hat{P_1}^ *)\subset [0,\rho^ 2-\epsilon]$.

\end{itemize}
\end{lemma}
\begin{proof} For (a) we first note that $\inf \Spec (\tilde{P_1}^ *\tilde{P_1})\geq \epsilon$. This is because any choice of a norm $1$ vector $\xi$ of $L$ identifies isometrically $L\otimes E$ with $E$ and $\tilde{P_1}$ becomes $P_{\xi}$ after this identification. 

Since $P_{\xi}$ is injective we get that $\tilde{P_1}$ is an isomorphism from $L \otimes E$ to $F_1$. Hence, using a standard equality \[\inf\Spec(\tilde{P_1}\tilde{P_1}^*)=\inf\Spec(\tilde{P_1}^ *\tilde{P_1})\geq \epsilon. \]

For (b) we note that if $Q:V\otimes E\ra L\otimes E$ denotes the orthogonal projection then 
\[\tilde{P}_1\tilde{P}_1^ *=P_1QP_1^ *\quad\mbox{ and }\quad\hat{P}_1\hat{P}_1^ *=P_1(1-Q)P_1^ *=\rho^ 2\id_{F_1}-\tilde{P}_1\tilde{P}_1^ * ,\] 
where the last equality used Lemma \ref{alglemm1}. Since $\sup\Spec(-\tilde{P}_1\tilde{P}_1^ * )\leq -\epsilon$ we get the claim. 
\end{proof}

We will need the following linear algebra inequality, also used in the next section.

\begin{lemma}\label{alglem}  Let $C:U\ra Y$ be a linear operator between inner product spaces. Suppose $U_2\subset U$ is a subspace  and $\hat{C}:=C\bigr|_{U_2}:U_2\ra Y$ satisfies $\Spec(\hat{C}\hat{C}^ *)\in [0,a]$ for some $a>0$. Let $u_2\in U_2$ and $u_1\in U$. Fix $c\geq 0$. Then
\begin{equation}\label{eq01}|u_2|^ 2+ c|C(u_1+u_2)|^ 2\geq \tilde{c}|C(u_1)|^ 2\end{equation}
where
\[ \tilde{c}=\left\{\begin{array}{ccc} \frac{1}{a} & \mbox {if} & C(u_1+u_2)=0 , \\
  \frac{c}{1+ac} &\mbox{if} & C(u_1+u_2)\neq 0 . \end{array}\right. \]

\end{lemma}

\begin{proof} First we have the inequality
\begin{equation} \label{eq02}a|u_2|^ 2\geq |C(u_2)|^ 2,\qquad \forall u_2\in U_2. \end{equation}
Indeed  $u_2\in U_2$ can be written as $u_2:=u_2'+u_2''$ where $u_2' \in \Imag \hat{C}^ *$ and $u_2''\in \Ker\hat{C}$. Then there exists $y\in Y$ such that $\hat{C}^ *(y)=u_2'$. Hence
\[ |C(u_2)|^ 2=|\hat{C}(u_2)|^ 2=|\hat{C}(u_2')|^ 2=|\hat{C}\hat{C}^ *(y)|^ 2=\langle (\hat{C}\hat{C}^ *)^ 2y,y\rangle\leq\]\[\leq  a\langle \hat{C}\hat{C}^* (y),y\rangle=a|\hat{C}^ *(y)|^ 2=a|u_2'|^ 2\leq a|u_2|^ 2.\]

Note that (\ref{eq02}) already implies (\ref{eq01}) when $C(u_1)=-C(u_2)$.

In general, for every $\mu\in (0,1]$, the  Peter-Paul inequality gives
\[|C(u_2)|^ 2= |C(u_1+u_2)-C(u_1)|^ 2\geq (1-\mu) |C(u_1)|^ 2+(1-\frac{1}{\mu})|C(u_1+u_2)|^ 2.\]
Hence, using (\ref{eq02}) we get
\[ |u_2|^ 2+ \frac{1}{a}\left(\frac{1}{\mu}-1\right)|C(u_1+u_2)|^ 2\geq \frac{1-\mu}{a}|C(u_1)|^ 2.\]
Let $c:=\frac{1}{a}\left(\frac{1}{\mu}-1\right)$ and note that $\frac{1-\mu}{a}=\tilde{c}$.
\end{proof}
We return to the injectively elliptic conformal projection $P:V\otimes E\ra F$.

\begin{corollary}\label{cord} Let $u=u_1+u_2$ in the orthogonal decomposition $L\otimes E\oplus W\otimes E$ and as before let $P=(P_1,P_2)$. Fix $c>0$.  Then
\begin{itemize}\item[(a)] We have \begin{equation}\label{eq03}|u_2|^ 2+c|P_1(u)|^ 2\geq \tilde{c}\epsilon|u_1|^ 2, \end{equation}
where 
\[\tilde{c}:=\left\{\begin{array}{ccc} \frac{1}{\rho^ 2-\epsilon}&\mbox{if} & P_1(u)=0 , \\
  \frac{c}{c(\rho^ 2-\epsilon)+1}& \mbox{if}&P_1(u)\neq 0 . \end{array}\right.\]
 \item[(b)] We have \[|u|^ 2+c|P(u)|^ 2\geq (1+\epsilon \tilde{c})|u_1|^ 2 ,\]
where 
\[\tilde{c}:=\left\{\begin{array}{ccc} \frac{1}{\rho^ 2-\epsilon}&\mbox{if} & P(u)=0 , \\
  \frac{c}{c(\rho^ 2-\epsilon)+1}& \mbox{if}& P(u)\neq 0 . \end{array}\right.\]
\end{itemize}

\end{corollary}
\begin{proof} For (a) we use  Lemma  \ref{alglem}  for $C:=P_1$, $U:=V\otimes E$ and $U_2:=W\otimes E$. By Lemma \ref{alglem0}, we have $\Spec(\hat{C}\hat{C}^ *)\subset [0,\rho^ 2-\epsilon]$. We thus have 
\[|u_2|^ 2+c|P_1(u_1+u_2)|^ 2\geq \tilde{c}|P_1(u_1)|^ 2.\]
But, since $\epsilon$ is the ellipticity constant we have
\[|P_1(u_1)|^ 2=|P(u_1)|^2\geq\epsilon |u_1|^ 2,\qquad\forall u_1\in L\otimes E. \]
This gives (a).

Then (b) is obtained from (a) by adding $|u_1|^ 2$ to both sides and noting two facts:
\begin{itemize}
\item[(1)] $|P(u)|^2\geq |P_1(u)|^ 2$. 
\item[(2)] $\frac{1}{\rho^ 2-\epsilon}> \frac{c}{c(\rho^ 2-\epsilon)+1}$ for all $c\geq 0$. Hence, if $P_1(u)=0$ but $P(u)\neq 0$,  one has by (\ref{eq03})
\[|u_2|^ 2+c|P(u)|^2\geq |u_2|^ 2+c|P_1(u)|^ 2\geq \epsilon \frac{1}{\rho^ 2-\epsilon} |u_1|^ 2\geq \frac{\epsilon c}{c(\rho^ 2-\epsilon)}|u_1|^ 2.\]
\end{itemize}
\end{proof}

We return now to first order differential operators $D:\Gamma(E)\ra \Gamma(F)$ which are conformal and injectively elliptic.  
\begin{lemma} \label{cord2} Let $\xi\in T^*_xM$ with $|\xi|=1$ and let $\phi\in \Gamma(E)$ be a section. Then $\xi\otimes \nabla_{\xi^{\sharp}}\phi$ is the orthogonal projection of $(\nabla \phi)_x\in T^ *_xM\otimes E_x$ onto $L\otimes E_x$ where $L=\langle \xi\rangle$. 
\end{lemma}
\begin{proof} For the orthonormal basis $\{e_1,e_2,\ldots,e_n\}$ of $T_x M$ with dual basis $\{de_1,de_2\ldots,de_n\}$ one writes
\[\nabla \phi=\sum_{i=1}^ n de_i\otimes \nabla_{e_i} \phi . \]
 The orthogonal projection onto $L\otimes E$ inside $T_x^*M\otimes E$ is $P_L\otimes \id_S$ where $P_L:T_x^*M\ra T_x^*M$ is the orthogonal projection onto $L$.

Then
\[P_{L\otimes E}(\nabla \phi)=\sum_{i=1}^ n P_{L}(de_i)\otimes \nabla_{e_i}\phi=\sum_{i= 1}^ n (\langle de_i,\xi\rangle \xi)\otimes \nabla_{e_i}\phi=\xi \otimes\nabla_{\xi^ {\sharp}}\phi.\]
\end{proof}

We use Corollary \ref{cord} in order to get the next result.
\begin{theorem} Let $D :\Gamma(E)\ra \Gamma(F)$ be a conformal, injectively elliptic, first order differential operator between Riemannian vector bundles $E,F \ra M$ with ellipticity constant $\epsilon$ and conformity factor $\rho$. Fix a constant $c\geq 0$. Let $\phi$
be a section of $E$ and let $x$ be a point in $M$ such that $\phi(x)\neq0$.
Then at $x$ the following pointwise inequality holds
\begin{equation}
|\nabla\phi|^{2}+c|D\phi|^{2}\geq\left(1+\tilde{c}\right)|d|\phi||^{2},
\end{equation} 
where
\[
\tilde{c}=\begin{cases}
\frac{\epsilon}{\rho^2-\epsilon} & if\,\,\,\left(D\phi\right)_{x}=0,\\
\frac{\epsilon c}{1+\left(\rho^2-\epsilon\right)c} & if\,\,\,\left(D\phi\right)_{x}\neq0.
\end{cases}
\]
\end{theorem}

\begin{proof} As in the proof of Theorem \ref{Dinjelliptic} we can assume that $d|\phi|^2\neq 0$ in a neighborhood of $x$. Let $u =(\nabla \phi)_x$, $\xi_0:=\frac{d|\phi|^2}{|d|\phi|^2|}(x)\in T^*_xM=V$,  $L=\langle\xi_0\rangle$  and $P$ the symbol of the operator $D$ at $x$. Then Corollary \ref{cord} together with Lemma \ref{cord2} gives
\[ |\nabla \phi|^ 2+ c|D(\phi)|^ 2\geq (1+\tilde{c})|\nabla_{\xi_0^ {\sharp}}\phi|^ 2 .\] 
But note that
\[|\nabla_{\xi_0^ {\sharp}}\phi|\geq|d|\phi|| , \]
as follows from the Cauchy-Schwarz applied to 
 \[ 2|\phi|\cdot|d|\phi||=|d|\phi|^2|=\langle d|\phi|^2,\xi_0\rangle=\xi_0^{\sharp} |\phi|^2=2\langle \nabla_{\xi_0^{\sharp}}\phi,\phi \rangle .
 \]
\end{proof}

\section{Extended Kato for the Hodge-de Rham operator} \label{sect4}

In this section we prove Theorem \ref{thm:hodge}, that is we turn our attention to $D=d+d^ *$. In fact, we will consider  the twisted version
 \[D:=d^{\nabla}+(d^ {\nabla})^ *:\Gamma(\Lambda^kT^ *M\otimes E)\ra \Gamma(\Lambda^ {k+1}T^ *M\otimes E\oplus \Lambda^ {k-1}T^ *M\otimes E) , \]
where 
\[d^{\nabla}(\omega\otimes s):=d\omega\otimes s+(-1)^{|\omega|} \omega\wedge \nabla^E s , \]
and $(d^{\nabla})^*$ is its formal adjoint.
 \begin{lemma}\label{nL} Let $\nabla^ {\otimes}:=\nabla^{LC}\otimes \id_E+\id_{\Lambda^kT^*M}\otimes \nabla^E$ be the connection on $\Lambda^kT^ *M\otimes E$. Then
\begin{equation}\label{eq001} (\epsilon\otimes \id_E)\circ \nabla^ {\otimes}=d^{\nabla} , \end{equation}
\begin{equation}\label{eq002} (-\iota\otimes \id_E)\circ \nabla^ {\otimes}=(d^{\nabla})^*,\end{equation}
where $\epsilon:T^*M\otimes \Lambda^kT^ *M\ra \Lambda^{k+1}T^ *M$ and $\iota:T^*M\otimes \Lambda^kT^ *M\ra\Lambda^{k-1}T^ *M$ are exterior and interior multiplication operators.
\end{lemma}
\begin{proof} We have
 \[ (\epsilon\otimes \id_E)(\nabla^{LC}\omega\otimes s+\omega\otimes \nabla^E s)=\frac{1}{\sqrt{k+1}}\left[\epsilon(\nabla^{LC}(\omega))\otimes s+(-1)^{|\omega|}\omega\wedge \nabla^{E}s\right].
 \]
 The reason for the sign $(-1)^{|\omega|}$ is that $\omega\wedge \nabla^{E}s\in \Lambda^kT^*M\otimes T^*M\otimes E$ while the operation $\epsilon$ wedges the $\Lambda^kT^*M$ and $T^*M$ of the tensor product but with the $T^*M$ component first. It is well known (e.g. see Prop 1.22 in \cite{BGV}) that
 \[\epsilon(\nabla^{LC}(\omega))=d\omega ,
 \]
and this proves (\ref{eq001}).  For (\ref{eq002}) fix a local  basis $\{ e_i\}$ for $T_xM$ with dual basis $de_i$. Then:
 \[ \nabla s=\sum_{i=1}^nde_i\otimes\nabla_{e_i}s ,
 \]
 and we can write (see also Prop 2.8 in \cite{BGV}):
 \[ (-\iota\otimes\id_E)(\nabla^{LC}\omega\otimes s+\omega\otimes \nabla^E s)=d^*\omega\otimes s-\sum_{i=1}^n\iota_{e_i}\omega\otimes \nabla_{e_i}^Es.
 \]
Notice that
 \begin{equation} \label{eq300} \left\langle\eta\otimes s_1,-\sum_{i=1}^n\iota_{e_i}\omega\otimes \nabla_{e_i}^Es_2\right\rangle=-\sum_{i=1}^n\langle de_i\wedge\eta,\omega\rangle\langle s_1,\nabla_{e_i}^Es_2\rangle=\end{equation}\[=\sum_{i=1}^n\langle de_i\wedge\eta,\omega\rangle \langle\nabla_{e_i}^Es_1,s_2\rangle-\langle [d\langle s_1,s_2\rangle]\wedge\eta,\omega \rangle=
 \]
 \[=\sum_{i=1}^n\langle de_i\wedge\eta \otimes \nabla^E_{e_i}s_1,\omega\otimes s_2\rangle- \langle [d\langle s_1,s_2\rangle]\wedge\eta,\omega \rangle=
 \]
 \[=(-1)^{|\eta|}\langle\eta\wedge\nabla^Es_1,\omega\otimes s_2\rangle- \langle [d\langle s_1,s_2\rangle]\wedge\eta,\omega \rangle.
 \]
 On the other hand, for sections with compact support since $d^*$ is the formal adjoint of $d$ we can write:
 \begin{equation}\label{eq400} \int_M \langle \eta\otimes s_1,d^*\omega\otimes s_2 \rangle=\int_M\langle s_1,s_2\rangle\langle \eta,d^*\omega\rangle=\int_{M} \langle d[\langle s_1,s_2 \rangle \cdot \eta ],\omega\rangle=\end{equation}\[=\int_M\langle d\eta \otimes s_1,\omega \otimes s_2\rangle+\int_M \langle [d\langle s_1,s_2\rangle]\wedge\eta,\omega \rangle.
 \]
 Integrating $(\ref{eq300})$ and adding it to $(\ref{eq400})$ proves that $(-\iota\otimes \id_E)\circ \nabla$ is the formal adjoint of $d^{\nabla}$ and therefore equals $(d^{\nabla})^*$.
\end{proof}

For ease of notation, we will denote  $d^{\nabla}$ by $d$ and $(d^{\nabla})^ *$ by $d^ *$ for the rest of this note.

If $T^ *M=L\otimes W$ is an orthogonal splitting where $L$ is a line bundle we correspondingly have

\[\Lambda^{k}T^*M=(L\wedge \Lambda^{k-1}W)\oplus \Lambda^{k}W ,
\]
where  $L\wedge \Lambda^{k-1}W\simeq L\otimes \Lambda^{k-1}W$. The first one is just the description "internal" to $\Lambda^{k}T^*M$ of the second one.

Consequently we have an orthogonal  decomposition of $T^*M\otimes \Lambda^{k}T^*M\otimes E$ into $4$ spaces
\begin{equation}\label{equ8} \underbrace{L\otimes (L\wedge \Lambda^{k-1} W)\otimes E}_{V_{11}} ~\oplus~ \underbrace{L\otimes \Lambda^{k}W\otimes E}_{V_{12}}~\oplus~ \underbrace{W\otimes (L\wedge \Lambda^{k-1}W)\otimes E}_{V_{21}}~\oplus~ \underbrace{W\otimes \Lambda^{k}W\otimes E}_{V_{22}}.
\end{equation}

 For ease of notation we will denote $\epsilon \otimes\id_E=:\epsilon^ E$ and $\iota\otimes \id_E=:\iota^ E$. 

  The next two results are straightforward checks and the proofs will be omitted. 

\begin{lemma}\label{declem} The following inclusions hold
\begin{eqnarray} \iota^E(V_{11})\subset \Lambda^{k-1}W\otimes E;\qquad \iota^E(V_{12})=0;\qquad \qquad\\ \iota^E(V_{21})\subset L\wedge \Lambda^{k-2}W\otimes E;\qquad \iota^E(V_{22})\subset \Lambda^{k-1}W\otimes E;  \\
\epsilon^E(V_{11})=0;\qquad \epsilon^E(V_{12})\subset L\wedge \Lambda^kW\otimes E;\qquad \qquad\\
 \epsilon^E(V_{21})\subset L\wedge\Lambda^kW\otimes E;\qquad \epsilon^E(V_{22})\subset \Lambda^{k+1}W\otimes E.
\end{eqnarray}

\end{lemma}

In fact we can say more:
\begin{lemma} With the trivialization $L\simeq \bR$ induced by a norm $1$ section $\xi_0$, one has the following bundle isometries:
 \[V_{11}\simeq \Lambda^{k-1}W\otimes E,\qquad V_{12}\simeq \Lambda^kW\otimes E.
 \] 
 Moreover, with these identifications, the operators $\iota^E$ and $\epsilon^E$ simplify as follows
 \begin{equation}\label{eq40}\iota^E\bigr|_{V_{11}}=\id_{\Lambda^{k-1}W\otimes E} , \end{equation}
\begin{equation}\label{eq41} \iota^E\bigr|_{V_{22}}:W\otimes \Lambda^kW\otimes E\ra \Lambda^{k-1}W\otimes E,\quad  \iota^E\bigr|_{V_{22}}(y\otimes z\otimes s)=\iota_{y^ {\sharp}}(z)\otimes s ,
\end{equation}
 \begin{equation}\label{eq4}\epsilon^E\bigr|_{V_{12}}=\id_{\Lambda^{k}W\otimes E} ,
 \end{equation} 
 \begin{equation}\label{eq42}\epsilon^E\bigr|_{V_{21}}: W\otimes  \Lambda^{k-1}W\otimes E \ra \Lambda^kW\otimes E,\qquad \epsilon^E\bigr|_{V_{21}}( y\otimes z\otimes s)= y\wedge z\otimes s .
 \end{equation} 
\end{lemma}

Following (\ref{equ8}), we  decompose $v\in T^ *M\otimes\Lambda^kT^*M\otimes E$ as
 \begin{equation}\label{equ101}v=v_{11}+v_{12}+v_{21}+v_{22}.\end{equation}

\begin{lemma} \label{lemiii}
Fix $c\geq 0$ and $c_*\geq 0$. 
\begin{itemize}
\item[(i)]
 Let $v_{12}\in V_{12}$ and $v_{21}\in V_{21}$. The following holds
 \[ |v_{21}|^2+c |\epsilon^E(v_{12}+v_{21})|^2\geq \tilde{c}|v_{12}|^2 , \]
 where
\[\tilde{c}=\left\{\begin{array}{ccc} \frac{1}{k}&\mbox{if} &  \epsilon^ E(v_{12}+v_{21})=0 , \\
 \frac{c}{1+ck}&\mbox{if}&   \epsilon^ E(v_{12}+v_{21})\neq 0 .
\end{array}  \right. \]
\item[(ii)] 
 Let $v_{11}\in V_{11}$ and  $v_{22}\in V_{22}$.
 The following holds:
 \[ |v_{22}|^2+c_* |\iota^E(v_{11}+v_{22})|^2\geq \tilde{c}_*|v_{11}|^2 , \]
 where 
\[\tilde{c}_*=\left\{\begin{array}{ccc} \frac{1}{n-k}&\mbox{if} &  \iota^ E(v_{11}+v_{22})=0 , \\
 \frac{c_*}{1+c_*(n-k)}&\mbox{if}&   \iota^ E(v_{11}+v_{22})\neq 0 .
\end{array}  \right. \]
\end{itemize}
\end{lemma}
\begin{proof} Both items rely on Lemma \ref{alglem} in conjunction with  Example \ref{firstex} items (2) and (3) if we observe that
\begin{itemize}
\item[(i)] $\epsilon^ E\bigr|_{V_{21}}$ is a conformal projection with conformity factor $\sqrt{k}$ because exterior multiplication is considered on forms of degree $k-1$.
We therefore use Lemma \ref{alglem} with $U=V_{12}\oplus V_{21}$, $C=\epsilon^ E$, $U_2=V_{21}$. Conclude via (\ref{eq4}).
\item[(ii)] $\iota_1^ E\bigr|_{V_{22}}$ is a conformal projection with conformity factor $\sqrt{n-k}$ since $W$ has rank $n-1$. 
We therefore use Lemma \ref{alglem} with $U=V_{11}\oplus V_{22}$, $C=\iota^ E$, $U_2=V_{22}$. Conclude via (\ref{eq40}).
\end{itemize}
\end{proof}

An immediate consequence is 
\begin{corollary} \label{coro1} If $v=v_{11}+v_{12}+v_{21}+v_{22}\in T^*M\otimes \Lambda^kT^*M\otimes E$ then
\[ |v|^2+c|\epsilon^E(v_{12}+v_{21})|^2+c_* |\iota^E(v_{11}+v_{22})|^2\geq (1+\min\{\tilde{c},\tilde{c}_*\})(|v_{11}|^2+|v_{12}|^2).
\]
\end{corollary}
\begin{proof} Since the decomposition  in (\ref{equ101}) is orthogonal we clearly have 
\[|v|^ 2=|v_{11}|^ 2+|v_{12}|^ 2+(|v_{22}|^ 2+ |v_{21}|^ 2).\] 
Then use Lemma \ref{lemiii}.
\end{proof}

We apply this corollary to $v=\nabla \phi$ where $\phi\in\Gamma(\Lambda^kT^ *M\otimes E)$ and note that, due to Lemma \ref{declem}, one also has the following straightforward inequalities:
\begin{equation}\label{equ1} |d\phi|^2=|\epsilon^E(\nabla\phi)|^2\geq  |\epsilon^E((\nabla \phi)_{12}+(\nabla \phi)_{21})|^2,
\end{equation}
\begin{equation}\label{equ2}|d^ *\phi|^2=|\iota^E(\nabla\phi)|^2\geq  |\iota^E((\nabla \phi)_{11}+(\nabla \phi)_{22})|^2.
\end{equation}

Note also that
 \begin{equation}\label{equ4}|(\nabla \phi)_{11}|^2+|(\nabla \phi)_{12}|^2=|\nabla_{\xi_0^{\sharp}}\phi|^2\geq |d|\phi||^ 2.
\end{equation}

Now, Corollary \ref{coro1}  immediately implies together with (\ref{equ1}), (\ref{equ2}) and (\ref{equ4}) the main result.

\begin{theorem} Let $\nabla$ be a metric connection on a Riemannian vector bundle $E \ra M^n$. Fix constants $c\geq 0$ and $c_{*}\geq 0$. Let $\phi$ be a section of $\Lambda^{k}T^{*}M\otimes E$ and let $x$ be a point in $M$ such that $\phi(x)\neq0$. Then at $x$ the following pointwise inequality holds
\[
|\nabla\phi|^{2}+c|d^\nabla \phi|^{2}+c_{*}|(d^{\nabla})^*\phi|^{2}\geq\left(1+\min\left\{ \tilde{c},\tilde{c}_{*}\right\} \right)|d|\phi||^{2},
\]
where
\[
\tilde{c}=\begin{cases}
\frac{1}{k} & if\,\,\,\left(d^\nabla \phi\right)_{x}=0,\\
\frac{c}{1+kc} & if\,\,\,\left(d^\nabla \phi\right)_{x}\neq0,
\end{cases}\,\,\,\,\,\,\,\,\,\,\tilde{c}_{*}=\begin{cases}
\frac{1}{n-k} & if\,\,\,\left((d^\nabla)^{*}\phi\right)_{x}=0,\\
\frac{c_{*}}{1+\left(n-k\right)c_{*}} & if\,\,\,\left((d^\nabla)^{*}\phi\right)_{x}\neq0.
\end{cases}
\]
On the left hand side, $d^\nabla$ is the covariant exterior derivative of
$\nabla$ and $(d^\nabla)^*$ is the adjoint of $d^\nabla$.
\end{theorem}

\section{Applications} \label{sect5}

In this section we give some applications of Theorem \ref{thm:foldo} and Theorem \ref{thm:hodge}.

We obtain a refined Kato inequality for Yang-Mills connections in
any dimension. This generalizes the refined Kato inequality in \cite{R1991, GKS2018}. We would like to point out that the proof in \cite{GKS2018}
is based on conformal metrics and Bochner's formula, a completely different method. Note that for Yang-Mills connections on $M^{n}$
with $n\geq5$ our constant is better than the constant in \cite{SU2019}.
\begin{corollary}
Let $A$ be a Yang-Mills connection on a Riemannian vector bundle $E \ra M^n$. Then the curvature $F$ of $A$ satisfies the pointwise inequality
\[
\left|\nabla F\right|^{2}\geq\left(1+\min\left\{ \frac{1}{2},\frac{1}{n-2}\right\} \right)\left|d\left|F\right|\right|^{2}.
\]
\end{corollary}
\begin{proof}
We have $dF=0$ and $dF^{*}=0$. Apply Theorem \ref{thm:hodge} to
$F$ with $k=2$. Note that this result is also a consequence of Theorem \ref{thm:foldo} or of Theorem \ref{Dinjelliptic}.
\end{proof}

We obtain an extended Kato inequality for Yang-Mills-Higgs pairs in
any dimension.
\begin{corollary}
Let $\left(A,\Phi\right)$ be a Yang-Mills-Higgs pair on a Riemannian vector bundle $E \ra M^n$. Fix a constant $a>0$. Then the curvature $F$ of $A$ and the Higgs field $\Phi$ satisfy the inequalities
\[
\left|\nabla F\right|^{2}\geq\left(1+\min\left\{ \frac{1}{2},\frac{a}{1+\left(n-2\right)a}\right\} \right)\left|d\left|F\right|\right|^{2}-a\left|\left[d\Phi,\Phi\right]\right|^{2},
\]
\[
\left|\nabla d\Phi\right|^{2}\geq\left(1+\min\left\{ \frac{1}{n-1},\frac{a}{1+a}\right\} \right)\left|d\left|d\Phi\right|\right|^{2}-a\left|\left[F,\Phi\right]\right|^{2}.
\]
\end{corollary}
\begin{proof}
(i) We have $dF=0$ and $d^{*}F=\left[d\Phi,\Phi\right]$ (we may
have $d^{*}F\neq0$). Apply Theorem \ref{thm:hodge} to $F$ with
$k=2$. (ii) We have $d^{*}d\Phi=0$ and $dd\Phi=\left[F,\Phi\right]$
(we may have $dd\Phi\neq0$). Apply Theorem \ref{thm:hodge} to $d\Phi$
with $k=1$.
\end{proof}
We obtain an extended Kato inequality for certain special Yang-Mills
connections (instantons) and special Yang-Mills-Higgs pairs (monopoles).
\begin{corollary}
Let $E \ra M^n$ be a Riemannian vector bundle.

(i) Suppose $n=4$. Let $A$ be a Yang-Mills connection on $E$ with curvature $F$. Write $F^{\pm}=\frac{1}{2}\left(F\pm*F\right)$
(note that $A$ is an instanton if and only if $F^{\pm}=0$). Then
\[
\left|\nabla F^{\pm}\right|^{2}\geq\left(1+\frac{1}{2}\right)\left|d\left|F^{\pm}\right|\right|^{2}.
\]

(ii) Suppose $n=3$. Let $\left(A,\Phi\right)$ be a Yang-Mills-Higgs
pair on $E$ with curvature $F$. Write $\omega^ {\pm}=*F\pm d\Phi$
(note that $\left(A,\Phi\right)$ is a monopole if and only if $\omega=0$).
Fix a constant $a>0$. Then
\[
|\nabla\omega^ {\pm}|^{2}\geq\left(1+\min\left\{ \frac{1}{2},\frac{a}{1+a}\right\} \right)\left|d\left|\omega^{\pm}\right|\right|^{2}-a\left|d\omega^ {\pm}\right|^{2}.
\]
We remark that $|d\omega^\pm| = |[\omega^\pm,\phi]|$.
\end{corollary}
\begin{proof}
(i) This was proved in \cite{GKS2018} by another method. We have
$dF^{\pm}=0$ and $d^{*}F^{\pm}=0$. Apply Theorem \ref{thm:hodge}
to $F$ with $n=4$ and $k=2$. (ii) In the special case $a=1$, this was
proved in \cite{F2023} by another method. We have $d^{*}\omega^ {\pm}=0$
(we may have $d\omega^ {\pm}\neq0$). Apply Theorem \ref{thm:hodge} with
$n=3$ and $k=1$. The remark follows by taking the exterior covariant derivative of $\omega^\pm$ and using the fact that $d^{*}F=\left[d\Phi,\Phi\right]$ and $dd\Phi=\left[F,\Phi\right]$.
\end{proof}
We recover the extended Kato inequality for the Dirac operator in
\cite{CJS2014}.
\begin{corollary}
Let $D=c\circ\nabla$ be a Dirac operator on a hermitian vector bundle $S \ra M^n$, where $\nabla$ is a metric connection on $S$ and $c:TM\otimes S\to S$ is a Clifford multiplication. Fix a constant $c>0$. Let $\phi$ be a section of $S$ and let $x$ be a point in $M$ such that
$\phi (x) \neq0$. Then at $x$ the following pointwise inequality holds
\[
\left|\nabla\phi\right|^{2}+c\left|D\phi\right|^{2}\geq\left(1+\tilde{c}\right)\left|d\left|\phi\right|\right|^{2},
\]
where
\[
\tilde{c}=\begin{cases}
\frac{1}{n-1} & if\,\,\,d\phi=0,\\
\frac{c}{1+(n-1)c} & if\,\,\,d\phi\neq0.
\end{cases}
\]
\end{corollary}

\begin{proof}
By Example \ref{firstex} the conformity factor satisfies $\rho^2 = n$, and it is well known that the ellipticity constant $\epsilon = 1$. Apply Theorem \ref{thm:foldo} with this $\rho$ and $\epsilon$.
\end{proof}

We close with another explicit computation.
\begin{corollary} Let $c:T^ *M\otimes S\ra S$ be a Clifford multiplication and $T:\Gamma(S)\ra  \Gamma(\Ker c)$ be the Penrose twistor operator with symbol the orthogonal projection onto $\Ker c$:
\[P:T^ *M\otimes S\ra \Ker c,\qquad P(v\otimes s):= v\otimes s+\frac{1}{n}\sum_{i=1}^ ne_i^*\otimes c_{e_i} c_v(s) . \]
 Then $\rho=1$, $\epsilon=\frac{n-1}{n}$ and therefore for any $\phi\in\Gamma(S)$ and any $x\notin \phi^ {-1}(0)$ one has
\[|\nabla \phi|^ 2+c|T(\phi)|^ 2\geq (1+\tilde{c})|d|\phi||^2 , \]
where $\tilde{c}=\left\{\begin{array}{ccc}n-1 &\mbox{if} & T(\phi)_x=0 , \\
\frac{(n-1)c}{n+c} &\mbox{if} & T(\phi)_x\neq 0 .
\end{array}\right.$
\end{corollary}
\begin{proof} The symbol $P$ is an orthogonal projection by definition hence $\rho=1$. One computes easily the adjoint of $P_v(s):=P(v\otimes s)$ to be $P_v^*\bigr|_{\Ker c}$ where $P_v^*:T^*M\otimes S\ra S$ is
\[P_v^ *(w\otimes s):=\langle v,w\rangle s+\frac{1}{n}c_vc_w(s) , \]
and therefore $P_v^ *P_v=\frac{n-1}{n}\id_S$. This fits with $C_v^ *C_v=\frac{1}{n}\id_S$ where $C:=\frac{1}{\sqrt{n}}c$ is the normalized Clifford multiplication that gives an orthogonal projection.  So $\epsilon=\frac{n-1}{n}$.
\end{proof}

\bibliographystyle{acm}

\vspace{0.3cm}
\end{document}